\documentclass[11pt,twoside,reqno]{amsart}
\usepackage{amsmath, amsfonts, amsthm, amssymb,amscd, graphicx, amscd}
\usepackage{float,epsf}
\usepackage[english]{babel}
\usepackage{enumerate}
\usepackage{tikz}
\usepackage{hyperref}
\usepackage{float} 
\usepackage{srcltx}
\usepackage{subfigure}
\usepackage{graphics}
\usepackage{epsfig}

\usepackage{geometry}
\usepackage{verbatim}
\usepackage{mathrsfs}

\newcommand{\R}{\mathbb{R}}

\newtheorem{theorem}{Theorem}[section]
\newtheorem{lemma}[theorem]{Lemma}

\newtheorem{proposition}[theorem]{Proposition}

\newtheorem{remark}[theorem]{Remark}

\newtheorem{question}[theorem]{Question}
\newtheorem{conjecture}{Conjecture}
\newtheorem*{remark*}{Remark}

\numberwithin{equation}{section}
\numberwithin{figure}{section}

\makeatletter
\@namedef{subjclassname@2020}{Mathematics Subject Classification 2020}
\makeatother

\renewcommand{\div}{\mathrm{div}} 

\def\intave#1{\int_{#1}\hbox{\llap{$\raise2.3pt\hbox{\vrule
				height.9pt width7pt}\phantom{\scriptstyle{#1}}\mkern-2mu$}}}

\everymath{\displaystyle}

\usepackage[misc]{ifsym}

\begin{document}

  \title{On location of maximum of gradient of torsion function}

\author[Q. Li]{Qinfeng Li}
\address{Qinfeng Li, School of Mathematics, Hunan University, Changsha, Hunan, China}
\email{liqinfeng1989@gmail.com}

\author[R. Yao]{Ruofei Yao}
\address{Ruofei Yao, School of Mathematics, South China University of Technology, Guangzhou, P.R. China.}
\email{yaoruofei@scut.edu.cn}

\thanks{Research of Qinfeng Li was supported by the National Science Fund of China for Youth Scholars (No. 12101215). Research of Ruofei yao was supported by the National Natural Science Fund of China for Youth Scholars (No. 12001543).}

\maketitle

\begin{abstract}
It has been a widely belief that for a planar convex domain with two coordinate axes of symmetry, the location of maximal norm of gradient of torsion function is either linked to contact points of largest inscribed circle or connected to points on boundary of minimal curvature. However, we show that this is not quite true in general. Actually, we derive the precise formula for the location of maximal norm of gradient of torsion function on nearly ball domains in $\mathbb{R}^n$, which displays nonlocal nature and thus does not inherently establish a connection to the aforementioned two types of points. Consequently, explicit counterexamples can be straightforwardly constructed to illustrate this deviation from conventional understanding. We also prove that for a rectangular domain, the maximum of the norm of gradient of torsion function exactly occurs at the centers of the faces of largest $(n-1)$-volume.
\end{abstract}
	
	\vskip0.2in

\section{Introduction}
\subsection{Background and Motivations}
Let $u=u_\Omega$ be the torsion function on a convex domain $\Omega \subset \mathbb{R}^n$. That is,
\begin{align}
\label{torsionfunction}
    \begin{cases}
        -\Delta u=1 \quad &\mbox{in $\Omega$,}\\
        u=0\quad &\mbox{on $\partial \Omega$.}
    \end{cases}
\end{align}
We are interested in finding the location of points on $\partial \Omega$ at which 
$|\nabla u|=-\partial u/\partial \nu$
achieves its maximum where $\nu$ is the unit outer normal vector on $\partial\Omega$. There are mainly two motivations.

First, the solution to $\eqref{torsionfunction}$ can be viewed as the steady state temperature in the thermal body $\Omega$, where the heat source is uniformly distributed in $\Omega$ with a unit value, the heat transfer occurs solely through conduction, and the external temperature is held at a constant value, specifically $0$. Then the quantity $\max_{\partial \Omega}|\nabla u|$ represents the maximal dispersion of temperature, and the location where the maximum is attained is thus called the \textit{points of maximal dispersion}. Finding the location of maximal dispersion is itself of interest. Moreover, in the very inspiring article \cite{Buttazzo}, the author introduces a thin thermal insulation problem for the purpose of finding the optimal distribution of insulating material in order to maximize the averaged heat of a given domain, and such problem recently gained renewed interest, as seen in the more recent articles \cite{BBN}, \cite{BBN1}, \cite{DNST}, \cite{HLL1} and \cite{HLL2}. In the insulation problem model, when the heat source is uniform, it is shown in \cite{HLL2} that for any domain which is not a ball, when the total amount of material is smaller than a threshold, then the insulation material is better not to cover the whole boundary. Moreover, an earlier result in \cite{ER03} shows that if the total amount of material goes to $0$, then the optimal distribution of material actually concentrates on points of maximal dispersion of solution to \eqref{torsionfunction}. Hence finding the points of maximal gradient on boundary can help design the optimal strategy for distributing insulation material. Unfortunately, \cite{ER03} does not delve into geometric feature of points of maximal dispersion, and later in \cite{BBN1}, the location of maximal dispersion is mentioned for annulus and squares, see also \cite{BBN}. 


Actually the points of maximal dispersion are exactly points of maximal stress in the torsion problem, and  finding out the location of such points is of significant physical interest and has been drawn attention much earlier, tracing back to Saint Venant's elasticity theory, see \cite{SV}, and this serves as our second motivation. The name torsion function of the solution to \eqref{torsionfunction} actually comes from elasticity theory, and $\Omega$ usually represents a planar convex domain which is the cross-section of the elastic bar. Roughly speaking, if $|\nabla u|$ becomes too large, then the material loses its elastic behaviour and becomes plastic. Hence the points at which $|\nabla u|$ achieves its maximum are also called \textit{fail points} or \textit{points of maximal stress}. For historical remarks on location of fail points, we refer to \cite{K87} and \cite{HS21} and references therein, see also the survey article \cite{KM93}.

Clearly, the maximal stress must occur on the boundary, due to maximum principle for subharmonic functions. It is in general difficult by theoretical method to decide where exactly on boundary the stress is maximal. Saint Venant conjectured that for a convex planar domain symmetric about two coordinate axes (we call this \textit{axially symmetric} later for convenience), the maximal stress occurs on the contact points of the largest inscribed circle.  This conjecture is not true in general, due to \cite{R90} and \cite{Sweer} via contradiction argument: either $\Omega$ or some super-level set of torsion function can serve as a counterexample. See also the counterexamples surveyed in \cite{Sweer2}.  In particular, a direct concrete counterexample was constructed in \cite{R902}, and there the author proposed the following modified conjecture which has been believed by many mechanical engineers: 
\begin{conjecture}
\label{rc}
For an axially symmetric planar convex domain centered at the origin, the fail points must either on the point of minimal distance to the origin, or points with minimal curvature.
\end{conjecture}
This conjecture is proved for the particular domain studied in \cite{R902}, but whether or not it is true in general is open, to the best of our knowledge.

As mentioned above, due to \cite{Sweer} and \cite{R90}, even for an axially symmetric convex domain, a definite correlation does not exist between the location of fail points and the positions on the boundary that are closest to the origin. It can also be imagined there should be no strict relation between location of fail points and points on boundary with minimal curvature. While there might not be an explicit reference addressing this matter, the reason is simple: considering a very narrow domain of two axes of symmetry, the left and right portions of boundary are flat, while the upper and lower portions of boundary are strictly concave. One can imagine (and can rigorously prove by asymptotic analysis) that when the domain is sufficiently narrow, the location of maximal gradient of torsion would be situated near the middle points in the upper and lower portions of boundary, while the curvature there is not zero and therefore not minimal. 

However, so far all counterexamples have only contradicted one scenario at a time, without effectively refuting both scenarios concurrently. Intuitively, fixing an axially symmetric convex domain, one can hardly say that fail points can neither occur on points of minimal curvature, nor occur on contact points of largest inscribed ball. Proving or disproving Conjecture \ref{rc} requires more precise analysis of the behavior of the gradient of torsion function on boundary, which is rather difficult for generic domains.

In fact, it is a widely held belief in elasticity theory that for any planar convex domain, the location of fail points should be near the contact points of largest inscribed disk, and of these at the one with minimal curvature, see for example the torsion chapter in \cite{Flugge}. However, this statement is vague and does not give information on how close fail points and contact points can be.

Some people believe that a stronger result holds, as stated in Conjecture \ref{rc}. In this direction, the most important result so far is due to \cite{K87}, which gives a positive answer under the additional assumption that the curvature of the boundary is decreasing in the first quadrant. In fact, under this condition, it is proved that $|\nabla u|$ is increasing along the boundary, and thus fail points lie exactly on end points of shorter axis, which are also contact points of largest inscribed circle and points of minimal curvature. A standard model in such case is elliptic domain. There is another very interesting article \cite{K98}, where the author proves that for any planar convex domain, the points of maximum curvature can never be fail points.

The monotonicity assumption of curvature in \cite{K87} is essential in applying moving plane method, which is the crucial technique of proving monotonicity of $|\nabla u|$. The moving plane method has been a very powerful tool to prove monotonicity of some functions related to PDEs with Dirichlet conditions, and thus the following question is also natural:

\begin{question}
    \label{q2}
Let $\Omega$ be an axially symmetric domain and $u$ be the torsion function associated to $\Omega$. If $|\nabla u|$ is increasing along $\partial \Omega$ in the first quadrant, is it true that the curvature of $\partial \Omega$ in the first quadrant must be decreasing?
\end{question}

Also, besides regular polygons and axially symmetric planar convex domains satisfying the monotonicity assumption of curvature in the first quadrant, we are not aware of any theoretical results proving the precise location of fail points in other domains, even absence for rectangular domains in literature.

Conjecture \ref{rc}, Question \ref{q2} and our curiosity of finding maximal heat dispersion, or maximal stress of material for other domains, motivate this present work.

\subsection{Our Results}

We disprove Conjecture \ref{rc} and also give a negative answer to Question \ref{q2}, both with explicit examples. 

Actually, instead of proving some specified points are not fail points, we can give a precise description of location of fail points for nearly ball domains, from which we demonstrate the following principle: given an axially symmetric convex domain in any dimension, both the location of contact points of largest inscribed ball and the distribution of curvature (in various sense) on the boundary, together do not in general relate to the location of fail points without additional assumptions.

Our approach is by perturbation argument which works in any dimension. Let
\begin{align}
\label{nearlyball}
    \Omega_t=\{(r,\sigma): r<1+t\zeta(\sigma)\},
\end{align}where $\sigma \in \mathbb{S}^{n-1}$ and $\zeta$ is a smooth function defined on $\mathbb{S}^{n-1}$. Hence $\Omega_t=F_t(B_1)$, where $B_1$ is the unit ball, and $F_t$ is some smooth diffeomorphism from $B_1$ to $\Omega_t$. The main ingredient is the establishment of the following pointwise formula for gradient of torsion on boundary:
\begin{align}
    \label{keyexpressionintro}
    |\nabla u(t)(y)|^2=\frac{1}{n^2}+\frac{2}{n^2}\Big(\zeta(\sigma)-\partial_\nu T\zeta(\sigma)\Big)t+O(t^2),
\end{align}where $u(t)$ is the torsion function on $\Omega_t$, $y=F_t(\sigma)\in \partial \Omega_t$, $\nu$ is the unit outer normal to $\partial B_1$, and $T\zeta$ is the harmonic extension of $\zeta$ into $B_1$. See Lemma \ref{crucial} for a proof. See also the remarks after Lemma \ref{crucial} for the validity of \eqref{keyexpressionintro} from other perspectives.

The pointwise formula \eqref{keyexpressionintro} is itself interesting, as one can see also the remarks after Lemma \ref{crucial} for the validity of \eqref{keyexpressionintro} from other perspectives. Going back to our consideration of maximal gradient of torsion function $u$, consequently, for $t>0$ small and domain given by  \eqref{nearlyball}, the location where $|\nabla u|$ takes its maximum, in polar coordinates, should be very close to the maximal point of the quantity \begin{align}
    \label{keyquantity}
\zeta-\partial_\nu T\zeta    
\end{align}
on $\partial B_1$. For a rigorous statement, see Theorem \ref{location}. However, neither the distance function nor curvature in polar coordinates is described by \eqref{keyquantity} in general. Actually, due to the term $\partial_\nu T\zeta$ which can be viewed as half Laplacian on $\zeta$, the location of maximal point of \eqref{keyquantity} is nonlocal and thus depends on global geometric property of domain, while both the location of minimal curvature and contact points of largest inscribed ball are local properties of domain without further global assumptions such as that in \cite{K87}. Hence we justify the principle mentioned above.

As a consequence, we prove the following with explicit examples (see Proposition \ref{faraway} and Proposition \eqref{monotonecounterexample}):
\begin{enumerate}
       \item There exists axially symmetric planar convex domains such that the points of maximal gradient of torsion function, the points of minimal curvature and the contact points of largest inscribed disk, are all different and even ``not close" to  each other.
    \item There exists axially symmetric planar convex domains such that along the boundary in the first quadrant, the norm of gradient of the torsion function is monotone while the curvature is not monotone.
\end{enumerate}

These disprove Conjecture \ref{rc} and answer Question \ref{q2}, and also suggest that the commonly held belief in elasticity theory mentioned before might not be quite true in general.

Besides the above results, we also exactly determine location of maximal norm of gradient of torsion function on rectangular domains, see Theorem \ref{rectangle}. As far as we are aware, this result is new. 

Last, we should mention that our understanding of location of maximal gradient of torsion function is still quite modest. For example we do not know the necessary and sufficient condition for an axially symmetric planar convex domain such that the gradient of torsion function is monotone in the first quadrant. We only somehow give the answer when the domain is nearly a ball. The complete solution seems open. We also propose some other questions which are open to us in section 6.

\section{Torsion function near a ball}
Let $\Omega_0=B_1$, the unit ball in $\mathbb{R}^n$. Consider the smooth transformation map
\begin{align}
    \label{diff}
F_t(x)=x+t\eta(x)+O(t^2),    
\end{align}where $\eta: \mathbb{R}^n \mapsto \mathbb{R}^n$ is a smooth vector field compactly supported in $\mathbb{R}^n$, and hence $F_t$ is a diffeomorphism for $|t|$ small. There are usually two typical choices. The first choice is simply $F_t(x)=x+t\eta(x)$. The second choice is regarding $F_t(x)=F(t,x)$, the flow map generated by $\eta$:
\begin{align*}
    \begin{cases}
    \frac{\partial}{\partial t}F(t,x)=\eta\circ F(t,x)\quad &x\in \mathbb{R}^n\\
    F(0,x)=x\quad & x\in \mathbb{R}^n
    \end{cases}
\end{align*}
By standard ODE theory, $F_t$ is also a smooth diffeomorphism satisfying \eqref{diff}.

Let $\Omega_t=F_t(B_1)$ and  $u(t)=u_{\Omega_t}$ be the torsion function in $\Omega_t$, that is
\begin{align*}
    \begin{cases}
        -\Delta u(t)=1 \quad &\mbox{in $\Omega_t$}\\
        u(t)=0\quad &\mbox{on $\partial \Omega_t$}
    \end{cases}
\end{align*}
We have the following lemma.
\begin{lemma}
    \label{crucial}
Let $\Omega_t, u(t)$ and $\eta$ be as above. Then $u(t)$ depends smoothly on $t$. Moreover, for $x \in B_1$ and $y=F_t(x)$, we have
\begin{align}
    \label{pointwisegrad}
|\nabla u(t)(y)|^2=\frac{|x|^2}{n^2}+\frac{2}{n^2}\Big(\eta \cdot x-x\cdot \nabla T\zeta\Big)t+O(t^2),
\end{align}where $\zeta=\eta \cdot \nu$ is defined on $\partial B_1$ and $T\zeta$ is the harmonic extension of $\zeta$ into $B_1$. 

In particular if $\sigma \in \partial B_1$, then we have
\begin{align}
\label{keyexpression}
    |\nabla u(t)(y)|^2=\frac{1}{n^2}+\frac{2}{n^2}\Big(\zeta(\sigma)-\partial_\nu T\zeta(\sigma)\Big)t+O(t^2),
\end{align}where $\nu$ is the unit outer normal to $\partial B_1$.
\end{lemma}

\begin{proof}
Let $\Tilde{w}=u(t)\circ F_t$, which is thus defined from $B_1$ to $\mathbb{R}$. Note that
\begin{align}
\label{inv}
   x= F_t ^{-1}(y)=y-t\eta(y)+O(t^2).
\end{align}
By \eqref{diff} and \eqref{inv}, and letting $\delta_{ij}$ be the Kronecker symbol, we have
\begin{align}
\label{t1}
    \frac{\partial x^k}{\partial y^i}(y)=\delta_{ik}-t\eta^k_i(y)+O(t^2)=\delta_{ik}-t\eta^k_i(x)+O(t^2).
\end{align}
Also, 
\begin{align}
\label{t2}
    \frac{\partial^2 x^k}{\partial (y^i)^2}(y)=-t\eta^k_{ii}(y)+O(t^2)=-t\eta^k_{ii}(x)+O(t^2).
\end{align}
In the following, repeated indexes means summation over the indexes, and we write $F_t^{-1}$ as $F_{-t}$. Hence
\begin{align*}
    \left(\Tilde{w}_k(F_{-t}(y)) \frac{\partial x^k}{\partial y^i}(y)\right)_i=\Delta u(t)(y).
\end{align*}
Hence
\begin{align}
\label{t3}
    \begin{cases}
        (\Tilde{w}_{kl}\circ F_{-t})\frac{\partial x^l}{\partial y^i}\frac{\partial x^k}{\partial y^i}+(\Tilde{w}_k \circ F_{-t})\frac{\partial x^k}{\partial (y^i)^2}=-1 \quad &\mbox{in $\Omega_t$}\\
        w=0\quad &\mbox{on $\partial \Omega_t$}
    \end{cases}
\end{align}
From \eqref{t1}-\eqref{t2}, 
\begin{align*}
    \frac{\partial x^k}{\partial y^i}(y)\frac{\partial x^l}{\partial y^i}(y)=\delta_{kl}-t\eta^l_k(x)-t\eta^k_l(x)+O(t^2),
\end{align*}and thus \eqref{t3} becomes
\begin{align}
\label{tildew}
    \begin{cases}
        \Delta \Tilde{w}-t(2\eta^k_l\Tilde{w}_{kl}+\Delta \eta^k\Tilde{w}_k)+O(t^2)=-1\quad &\mbox{in $B_1$}\\
        w=0\quad &\mbox{on $\partial B_1$}
    \end{cases}
\end{align}
By standard elliptic regularity theory and implicit function theorem, $\Tilde{w}(x,t): B_1\times (-1,1)\mapsto \mathbb{R}$ is a smooth function, and hence $u(t)$ is smoothly depending on $t$. We write 
\begin{align}
\label{expandw}
    \Tilde{w}(x,t)=w(x)+tv(x)+O(t^2),
\end{align}and thus from \eqref{tildew}, $v$ satisfies
\begin{align}
    \label{equationforv'}
\begin{cases}
    \Delta v-2\eta^k_l w_{kl}-\Delta \eta^k w_k=0\quad &\mbox{in $B_1$}\\
    v=0\quad &\mbox{on $\partial B_1$}
\end{cases}
\end{align}
Note that $w$ is exactly the torsion function in $B_1$, and thus 
\begin{align*}
    w(x)=\frac{1-|x|^2}{2n}.
\end{align*}
Plugging this into \eqref{equationforv'}, we eventually obtain
\begin{align}
    \label{vequation}
\begin{cases}
    \Delta v+\frac{2}{n}\div \eta+\frac{x\cdot \Delta \eta}{n}=0\quad &\mbox{in $B_1$}\\
    v=0\quad &\mbox{on $\partial B_1$}
\end{cases}
\end{align}
In the following, the derivatives are respect to the $x$-variable. Note that
\begin{align*}
    x\cdot \Delta \eta =&(x^k\eta^k_i)_i-x^k_i\eta^k_i\\
    =& \left((x^k\eta^k)_i-x^k_i\eta^k\right)_i-\delta_{ik}\eta^k_i\\
    =& \Delta (\eta \cdot x)-2\div \eta
\end{align*}
Hence \eqref{vequation} becomes
\begin{align}
    \label{vequation2}
\begin{cases}
    \Delta v+\frac{1}{n}\Delta (\eta \cdot x)=0\quad &\mbox{in $B_1$}\\
    v=0\quad &\mbox{on $\partial B_1$}
\end{cases}
\end{align}
Hence 
\begin{align}
    \label{vexpression}
v=\frac{1}{n}(T\zeta-\eta\cdot x),    
\end{align}
where $\zeta$ is the restriction of $\eta \cdot x$ on $\partial B_1$ and $T\zeta$ is the harmonic extension of $\zeta$ to $B_1$. 

Now we evaluate $|\nabla u(t)|^2$ at  $y \in \Omega_t$. First, we have
\begin{align}
\label{nabla1}
    |\nabla u(t)|^2=|\nabla (\Tilde{w}\circ F_{-t})|^2=\sum_j(\sum_i \Tilde{w}_i\circ F_{-t} \frac{\partial x_i}{\partial y^j})^2
\end{align}
In view of \eqref{t1} and \eqref{expandw}, we have from \eqref{nabla1} that for $y=F_t(x)$ with $x \in B_1$, 
\begin{align}
    \label{nabla2}
|\nabla u(t) (y)|^2=|\nabla w|^2(x)+t\left(2w_j(x)v_j(x)-2w_i(x)w_j(x)\eta^i_j(x)\right)+O(t^2).    
\end{align}
By \eqref{vexpression}, we eventually have
\begin{align*}
    |\nabla u(t)|^2(y)=&|\nabla w|^2(x)+t\Big(-\frac{2}{n^2}\left(x\cdot \nabla T\zeta-(\eta^ix^i)_jx^j\right)-\frac{2}{n^2}\eta^i_jx^ix^j\Big)+O(t^2)\\
    =&\frac{|x|^2}{n^2}+t\frac{2}{n^2}(\eta \cdot x-x \cdot \nabla T\zeta)+O(t^2).
\end{align*}
Hence for $\sigma \in \partial B_1$, and thus for $y=F_t(\sigma)\in \partial \Omega_t$, 
\begin{align*}
    |\nabla u(t)(y)|^2=\frac{1}{n^2}+t\frac{2}{n^2}\Big(\zeta(\sigma)-\partial_\nu T\zeta\sigma\Big)+O(t^2),
\end{align*}where $\nu$ is the unit outer normal to $\partial B_1$.
\end{proof}

We give some remarks which justify \eqref{keyexpression} from different perspectives. 

First, one can readily see from \eqref{keyexpression} that the first order term of $|\nabla u(t)|^2$ depends only on the normal component of the vector field $\eta$ along $\partial B_1$. For example, if $F_t$ is a translation, then $\eta$ is a constant vector field, and thus $\zeta(x)$ is linear combinations of coordinate functions. In such case, $\zeta-\partial_\nu T\zeta\equiv 0$ on $\partial B_1$. This coincides with the fact that the translation of domain does not affect the distribution of the torsion function. 

Second, if $F_t$ is volume preserving, then \begin{align}
\label{int0}
    \int_{\partial B_1} \zeta \, d\sigma=0.
\end{align}In view that $T\zeta$ is harmonic, we have
\begin{align*}
    \int_{\partial B_1}(\zeta-\partial_\nu T\zeta)\, d\sigma=0.
\end{align*}
Hence $\zeta-T_\nu\zeta$ has strictly positive and strictly negative part on $\partial B_1$, unless $\zeta$ is an eigenfunction of first nonzero Stekl\"off eivenvalue, which must be coordinate functions and thus $F_t$ then becomes a translation map.  Now let $f(t)=\sup_{x \in \partial \Omega_t}|\nabla u(t)|^2$, then from \eqref{keyexpression} and the above comments, as long as $F_t$ is not a translation mapping, $|t|\ne 0$ and $|t|$ small, then $|\nabla u(t)|^2$ is strictly bigger than $|\nabla u(0)|^2$, and hence $f$ takes a strict local minimum at $t=0$. This exactly says the following: 
\begin{remark}
\label{localop}
Each shape $\Omega$ corresponds to a maximal dispersion of temperature (or maximal stress) defined by $\sup_{\partial \Omega}|\partial_\nu u_\Omega|=\sup_{\partial \Omega}|\nabla u_\Omega|$, where $u_\Omega$ is the torsion function associated to $\Omega$. Then among all shapes of fixed volume, ball locally has least maximal dispersion (or least maximal stress). 
\end{remark}
We mention that ball is not a global minimizer to the functional $\tau(\Omega):=\sup_{\partial \Omega}|\nabla u_\Omega|$ among smooth domains with fixed volume. For example, if $\Omega$ is a thin ellipse with the same area as the unit disk, then $\tau(\Omega)$ is close to $0$. The maximizer among domains with fixed volume for $\tau(\cdot)$ is still open, and some partial progress can be found in \cite{HS21} and references therein. 

Third, let us still consider volume preserving transformation. If one takes integration of the more general form \eqref{pointwisegrad} over $\Omega_t$, and using that 
\begin{align*}
    det\left(\nabla F_t(x)\right)=1+\div \eta(x) t+O(t^2),
\end{align*}from area formula one has
\begin{align*}
    \int_{\Omega_t} |\nabla u(t)|^2 (y) \, dy=\int_{B_1} \frac{|x|^2}{n^2}\, dx+t\int_{B_1} \left(\frac{|x|^2}{n^2}\div \eta+\frac{2}{n^2}(\eta \cdot x-x\cdot \nabla T\zeta)\right) \, dx+O(t^2).
\end{align*}
By \eqref{int0} and integration by parts, 
\begin{align*}
    \int_{B_1}  \left(\frac{|x|^2}{n^2}\div \eta+\frac{2}{n^2}\eta \cdot x\right)\, dx=\int_{\partial B_1}\frac{1}{n^2}\zeta \, d\sigma=0.
\end{align*}Also, since $x=\nabla (|x|^2/2)$ and $T\zeta$ is harmonic, 
\begin{align*}
    \int_{B_1} x\cdot \nabla T\zeta \,dx=\int_{\partial B_1} \frac{1}{2}\partial_\nu T\zeta\, d\sigma=0.
\end{align*}
Hence we derive that 
\begin{align*}
    \frac{d}{dt}\Big|_{t=0}\left(\int_{\Omega_t} |\nabla u(t)|^2 (y) \, dy\right)=0.
\end{align*}This says that fixing volume of domains, ball is a critical shape of the functional 
\begin{align*}
    T(\Omega):=\int_{\Omega}|\nabla u_\Omega|^2 \, dx=\int_\Omega u_\Omega \, dx,
\end{align*}where $u_\Omega$ is the torsion function on $\Omega$. This matches the seminal Saint-Venant inequality, which actually says that ball is minimizer.

To end this section, we mention that it would also be interesting to derive all higher order terms of $|\nabla u(t)|^2$ if the transformation map is the flow map induced by ODE, though it is not the purpose of our work here. Second, similar expansion for Robin boundary conditions can also be considered.

\section{Location of maximal gradient of torsion function and minimal curvature when domain is close to a disk}

In this section, we consider convex domains which are close to a ball in $\mathbb{R}^n$. Particularly, using polar coordinates, we consider 
\begin{align}
    \label{polardomain}
\Omega_t=\{(r, \sigma): r<1+t\zeta(\sigma)\},     
\end{align}
where $\sigma\in \mathbb{S}^{n-1}$ and $\zeta$ is a smooth function defined on $\mathbb{S}^{n-1}$. For each $x\in B_1$, we write $x=(r,\sigma)$, and consider the smooth transformation
\begin{align}
    \label{ht}
F_t(x)=(r+t\chi(r)\zeta(\sigma),\sigma),   
\end{align}where $\chi(r)$ is chosen to be a smooth increasing function such that
\begin{align*}
    \chi(r)=0\quad \mbox{if $r\le 1/4$}\quad \mbox{and}\quad \chi(r)=1 \quad \mbox{if $r\ge 3/4$}
\end{align*}
Hence $\Omega_t$ is exactly $F_t(B_1)$, for $|t|$ small. Then as a consequence of Lemma \ref{crucial}, we have the following theorem, which gives location of points of maximal gradient of torsion function, when the domain is close to a ball. 
\begin{theorem}
    \label{location}
Let $u(t)$ be the torsion function associated to $\Omega_t$ described by \eqref{polardomain}, 
\begin{align}
    \label{gradsquare}
E=E(\sigma,t):=|\nabla u(t)|^2(1+t\zeta (\sigma), \sigma) 
\end{align}and 
\begin{equation}
\label{mathcalF}
\mathcal{F}(\sigma)=\zeta(\sigma)-\frac{\partial (T\zeta)}{\partial {r}}(1, \sigma),
\end{equation}where $T\zeta$ is the harmonic extension of $\zeta$ into the unit ball $B_1$. If $\hat{\sigma}$ is a unique maximum point of $\mathcal{F}(\cdot): ~ U\subset \mathbb{S}^{n-1}\to \R$ which is non-degenerate, then there exist $t_0>0$ and a smooth function $\hat{\varsigma}: ~ (-t_{0}, t_{0}) \to \mathbb{S}^{n-1}$ with $\hat{\varsigma}(0)=\hat{\sigma}$ such that 
for $0<t<t_{0}$, $\hat{\varsigma}(t)$ is the unique (non-degenerate) global maximum point of $E=E(\cdot, t)$ (i.e., $|\nabla u(t)|$ on $\partial\Omega_{t}$) over $U$.
\end{theorem}

\begin{proof}
    According to Lemma \ref{crucial}, we have
\begin{align}
\label{expa}
    E(\sigma, t)=\frac{1}{n^2}+\frac{2}{n^2}\left(\zeta(\sigma)-\frac{\partial(T\zeta)}{\partial r}(1,\sigma)\right)t+O(t^2)
\end{align}
Define 
\begin{align}
\label{fuzhu2}
    \Tilde{E}(\sigma_,t):=\begin{cases}
        \frac{E(\sigma,t)-\frac{1}{n^2}}{t}\quad &t\ne 0\\
        \frac{2}{n^2}\mathcal{F}(\sigma)\quad &t=0.
    \end{cases}
\end{align}
Then $\Tilde{E}:\, \mathbb{S}^{n-1}\times (-1,1)\mapsto \mathbb{R}$ is a smooth function. `If $\hat{\sigma}$ is a critical point of $\mathcal{F}(\cdot)$, then 
\begin{align*}
    \nabla_{\mathbb{S}^{n-1}}\Tilde{E}(\sigma,0)=\frac{2}{n^2} \nabla_{\mathbb{S}^{n-1}}\mathcal{F}(\sigma)=0.
\end{align*}
If $\hat{\sigma}$ is further to be nondegenerate, then by implicit function theorem, in a neighborhood of $(\hat{\sigma}, 0)$, $\nabla_{\mathbb{S}^{n-1}}\Tilde{E}(\sigma, t)=0$ is uniquely solved by a smooth function $\sigma=\hat{\varsigma}(t)$. If $\hat{\sigma}$ is a unique maximum point of $\mathcal{F}(\cdot)$ over an open subset $U$ of $\mathbb{S}^{n-1}$, then due to \eqref{fuzhu2}, for $t>0$ small, $\hat{\varsigma}(t)$ is the unique maximum point of $\Tilde{E}(\cdot, t)$ over $U$, and hence for $t>0$ small, $\hat{\varsigma}(t)$ is the unique maximum point of $E(\cdot, t)$ over $U$.
\end{proof}

As application for planar domains, we may regard $\zeta$ as a function of the polar angle $\theta\in [0,2\pi)$. The next proposition gives location of minimal curvature of a planar domain close to a disk.

\begin{proposition}
\label{locationcurvature}
Let $\Omega_t=\{(r, \theta): r<1+t\zeta(\theta)\}$. Then the curvature $\kappa=\kappa(\theta, t)$ of $\partial\Omega_{t}$ is given by 
\begin{equation*}
	\kappa(\theta, t)=\frac{(1+t \zeta)^2 + 2(t \zeta')^2 - (1+t \zeta)(t \zeta'')}{[(1+t \zeta)^2 + (t \zeta')^2]^{3/2} }.	
\end{equation*}
If $\bar{\theta}\in\mathbb{S}^{1}$ is a unique  global maximum point of $\zeta+\zeta''$ over an open subset $U$ of $\mathbb{S}^{1}$ and it is non-degenerate, then there exists a constant $\bar{t}_{0}>0$ and a smooth function $\bar{\vartheta}: ~ (-\bar{t}_{0}, \bar{t}_{0}) \to \mathbb{S}^{1}$ with $\vartheta(0)=\bar{\theta}$ such that for $0<t<\bar{t}_{0}$, $\theta=\bar{\vartheta}(t)$ is the unique global minimum point (which is non-degenerate) of $\kappa(\cdot, t)$ over $U$.
\end{proposition}

\begin{proof}
	The expression of the curvature $\kappa(\theta, t)$ is standard.  It is clear that $\kappa(\theta, t)$ is a smooth function in $(\theta, t)$, and by direct computation,
 \begin{align*}
     \kappa(\theta,t)=1-(\zeta+\zeta'')t+O(t^2).
 \end{align*}
	Denote 
	\begin{equation*}
		{K}(\theta, t) = 
		\begin{cases}
		\frac{\kappa(\theta, t) - 1}{t} &\text{ if } t \neq 0,
		\\
		-(\zeta+\zeta'')(\theta) &\text{ if } t \neq 0.	
		\end{cases}
	\end{equation*}
It is clear that ${K}: ~ \mathbb{S}^{1}\times (-1, 1) \rightarrow \R$ is a smooth function, and 
\begin{align*}
	\frac{\partial{K}}{\partial\theta}(\theta, 0)
	= - (\zeta+\zeta'')', \quad 
	\frac{\partial^2{K}}{\partial\theta^2}
	= - (\zeta+\zeta'')''.
\end{align*}
Suppose that
$(\zeta+\zeta'')'=0$ and  $(\zeta+\zeta'')''\neq0$ at $\theta=\bar{\theta}$. 
%
%
By applying implicity function theroem to $\partial{K}/\partial\theta$, 
in a neiborhood of $(\theta, t)=(\bar{\theta}, 0)$, the equation $\partial_{\theta}{K}(\theta, t) = 0$ is uniquely solved by a smooth function $\theta=\bar{\vartheta}(t)$ with $\bar{\vartheta}(0)=\bar{\theta}$. Moreover, if $\bar{\theta}\in \mathbb{S}^1$ is a unique global maximum point (which is non-degenerate) of $\zeta+\zeta''$ over $U$, then 
\begin{equation*}
{K}(\theta, t)>{K}(\bar{\vartheta}(t), t) \text{ for } \theta\in U \backslash\{\bar{\vartheta}(t)\}
\end{equation*}
for $|t|<\bar{t}_{0}$ for some $\bar{t}_{0}>0$. Therefore, $\theta=\bar{\vartheta}(t)$ is the unique global minimum point (which is non-degenerate) of $\kappa(\cdot, t)$ over $U$ for $0<t<\bar{t}_{0}$.
\end{proof}

\begin{remark}
\label{trigseries}
	If $\zeta(\theta) = \sum_{n>0}{c}_{n}\cos(n\theta)$,
	then $(\mathcal{T}\zeta)=\zeta(\theta) = \sum_{n>0}{c}_{n}r^{n}\cos(n\theta)$, and 
	\begin{align*}
		(\zeta+\zeta'')(\theta) =& \sum_{n>1}{c}_{n}(1-n^2)\cos(n\theta),
		\\
		\mathcal{F}(\theta) =& \sum_{n>1}{c}_{n}(1-n)\cos(n\theta).
	\end{align*}
Also, if $\Omega_t \subset \mathbb{R}^2$ given by \eqref{polardomain} is axially symmetric, then $\zeta$ must be of the form
\begin{align*}
    \zeta(\theta)=\sum_{k\ge 1} c_{2k}\cos(2k\theta)
\end{align*}
\end{remark}

\section{Counterexamples}
In this section we disprove Conjecture \ref{rc} and answer Question \ref{q2} mentioned in the introduction.

\begin{proposition}
\label{faraway}
Let $\Omega=\{x\in\R^2: ~ r<1+t\zeta(\theta)\}$ be a smooth domain with two coordinate axes of symmetry where $\zeta(\theta) = -4\cos(2\theta) + \cos(4\theta)$. Set $\Gamma^{++}=\partial\Omega\cap\{x_{1}\geq0, x_{2}\geq0\}$. Then for  $0<t\ll1$,
\begin{enumerate}[$(1)$]
	\item 
	the maximum of $|\nabla u(\cdot, t)|$ on $\Gamma^{++}$ is achieved at exactly one point $\bar{P}_{t}$ with corresponding $\theta=\bar{\vartheta}(t) = \bar{\theta} + O(t)$ where $\bar{\theta}=\arccos(\sqrt{2/3})$;
	\item 
	the minimum of the curvature $\kappa(\cdot, t)$ of $\Gamma^{++}$ is achieved at exactly one point $\hat{P}_{t}$ with corresponding $\theta=\hat{\vartheta}(t) = \hat{\theta} + O(t)$ where $\hat{\theta}=\arccos(\sqrt{3/5})$;
	\item 
	the minimum of the distance function $x\in\Gamma^{++}\mapsto |x|$ is achieved at exactly one point $\tilde{P}_{t}$ with corresponding $\theta=0$, which corresponds to end of shorter axis.
\end{enumerate}
\end{proposition}

\begin{proof}
If 
$\zeta(\theta) = -4\cos(2\theta) + \cos(4\theta)$, then by Remark \ref{trigseries}, 
\begin{align*}
	(\zeta+\zeta'')(\theta) =&12\cos(2\theta)-15\cos(4\theta),
	\\
	\mathcal{F}(\theta) =&4\cos(2\theta)-3\cos(4\theta).
\end{align*}
Note that $\zeta$ is symmetric with respect to both $\theta=0$ and $\theta=\pi/2$ ($\Omega_{t}$ is symmetric with respect to both coordinate axes).
Restricted on $[0, \pi/2]$, $(\zeta+\zeta'')$ has a unique global maximum point $\bar{\theta}=\arccos(\sqrt{3/5})$ (it is non-degenerate),  $\mathcal{F}$ has a unique global maximum point $\hat{\theta}=\arccos(\sqrt{2/3})$  (it is non-degenerate), and $\zeta$ has a global minimum point at $\theta=0$. By Theorem \ref{location} and Proposition \ref{locationcurvature}, we conclude the results.
\end{proof}

\begin{remark}
 For the above example, since $t>0$ is small, the closeness between points on boundary can measured by polar angle $\theta$, and in the first quadrant, the angle between the unique point of maximal gradient of torsion function, and the unique contact point of largest inscribed circle, are roughly 35 degree, which we do not think they are "near" each other. Hence the common belief that points of maximal stress are near contact points of largest inscribed circle, may not be quite true or should be further clarified.
\end{remark}

\begin{remark}
    In fact, in the previous example, restricted in the first quadrant, $\zeta$ is increasing, while $|\nabla u|$ is not monotone. Hence the monotone condition of curvature in \cite{kawohl1987location} cannot be replaced by the monotone condition of distance to the origin in order for the monotonicity of $|\nabla u|$.
\end{remark}

\begin{proposition}
\label{monotonecounterexample}
    Let $\Omega=\{x\in\R^2: ~ r<1+t\zeta(\theta)\}$ be a smooth domain with two coordinate axes of symmetry where $\zeta(\theta) = 13\cos(2\theta) - \cos(4\theta)$ and $0<t\ll1$. 
	Then there exits $t_0>0$ such that when $0<t<t_0$, $|\nabla u|$ is strictly monotone along $\Gamma^{++}=\partial\Omega\cap\{x_{1}\geq0, x_{2}\geq0\}$, while the curvature of $\partial\Omega$ is not monotone along $\Gamma^{++}$.

\end{proposition}

	\begin{proof}		
		From Theorem \ref{location} and Remark \ref{trigseries},
		\begin{align*}
			E(\theta,t):=|\nabla u|^2(\theta, t)=\frac{1}{4}+\frac{1}{2}\mathcal{F}(\theta) t+O(t^2).
		\end{align*}
		and
		\begin{align*}
			(\zeta+\zeta'')(\theta) =&-39\cos(2\theta)+15\cos(4\theta),
			\\
			\mathcal{F}(\theta) =&-13\cos(2\theta)+3\cos(4\theta).
		\end{align*}
		By the symmetry of domain $\Omega$ (and hence $u$), we know $E$ is symmetric with respect to both $\theta=0$ and $\theta=\pi/2$, 
		\begin{align}\label{eq402a}
			\partial_{\theta}E(\theta, t) = 0, \quad \partial_{\theta t}E(\theta, t) = 0 \text{ at } \theta=0, \pi/2.			
	\end{align}
Note that 
	\begin{gather}
	\label{eq403a}
	\partial_{\theta t}E(\theta, t) = \sin(2\theta)(13-12\cos(2\theta)) > 0   \text{ for } \theta \in(0, \tfrac{\pi}{2}) \text{ and } t=0,
	\\ \label{eq403b}
	\partial_{\theta\theta t}E(0, 0) = 2, \quad \partial_{\theta\theta t}E(\tfrac{\pi}{2}, 0) = -50.
\end{gather}
From \eqref{eq403b}, there exists a small positive constant $\delta_{0}\in(0, 1/2)$ and $t_{1}>0$ such that when $|t|<t_{1}$,
\begin{align*}
	\partial_{\theta\theta t}E(\theta, t) > 0 \, \mbox{for}\,  |\theta|<\delta_{0} \text{ and } \partial_{\theta\theta t}E(\theta, t) < 0 \text{ for } |\theta-\tfrac{\pi}{2}|<\delta_{0}. 
\end{align*}
Combining this with \eqref{eq402a}, 
	\begin{align*}
	\partial_{\theta t}E(\theta, t) > 0   \text{ for } \theta \in(0, \delta_{0})\cup(\tfrac{\pi}{2}-\delta_{0}, \tfrac{\pi}{2}) \text{ and } |t|<t_{1}.
\end{align*}
By \eqref{eq403a}, there exists $t_{2}>0$ such that  
	\begin{align*}
	\partial_{\theta t}E(\theta, t) > 0   \text{ for } \theta \in [\delta_{0}, \tfrac{\pi}{2}-\delta_{0}] \text{ and } |t|<t_{2}.
\end{align*}
Therefore, setting $t_{0}=\min\{t_{1}, t_{2}\}$,
\begin{align*}
	\partial_{\theta t}E(\theta, t) > 0   \text{ for } \theta \in (0, \tfrac{\pi}{2}) \text{ and } |t|<t_{2}.
\end{align*}
Combining this with the fact $\partial_{\theta}E(\theta, 0)\equiv0$, we get 
\begin{align*}
	\partial_{\theta}E(\theta, t) > 0   \text{ for } \theta \in (0, \tfrac{\pi}{2}) \text{ and } 0<t<t_{0}.
\end{align*}
It then follows that $|\nabla u|(\theta,t)$ is strictly increasing with respect to $\theta\in[0, \tfrac{\pi}{2}]$ when $0<t<t_0$.

However, direct computation shows that $\zeta+\zeta''$ is not monotone for $\theta \in [0,\pi/2]$. This implies that the curvature of $\partial\Omega$ is not monotone along $\Gamma^{++}$ whenever $t$ is small enough.
	\end{proof}

\section{Location of maximal gradient of torsion function on rectangular domains}

\begin{theorem}
\label{rectangle}
    Let $\Omega$ be a rectangular domain. Then the location of maximal norm of gradient of torsion function occurs at the middle points of the longer sides. 
\end{theorem}

	\begin{proof}
		
		Let $\Omega=(-L,L)\times (-l,l)$, with $L\ge l$, and let $u$ be the torsion function on $\Omega$. It is obvious that $u$ is symmetry with respect to both coordinates axes. Moreover, $u$ satisfies the monotone properties in half domain,
		\begin{align*}
		u_{x} < 0 \text{ in } \Omega\cap\{x>0\} \text{ and } u_{y} < 0 \text{ in } \Omega\cap\{y>0\},
		\end{align*}
	where it can be obtained by applying maximum principle to directional derivative direction, see also the well-known result of moving plane method \cite{GNN}.
From the boundary condition of $u$, we know that $u_{x}=0$ on $(-L, L)\times\{l\}$ and $u_y=0$ on $\{L\} \times (-l,l)$. Applying the Hopf lemma to the harmonic function $u_{x}$ and $u_{y}$, we get that 
\begin{align*}
	u_{xy} > 0 \text{ on both } (0, L)\times\{l\} \text{ and } \{L\}\times(0, l).
\end{align*}
It then follows that on $[0, L]\times\{l\}$,  $|\partial_{\nu}u|=|u_{y}|=-u_{y}$ is strictly decreasing with respect to variable $x$, while on $\{L\}\times [0,l]$,  $|\partial_{\nu}u|=|u_{x}|=-u_{x}$ is strictly decreasing with respect to variable $y$. 
This, in turn, implies that $|\nabla u|$ achieves its maximum in the middle point of side.

It remains to determine the exactly location of global maximum of $|\nabla u|$, i.e., we need to compare $|\nabla u|(0,l)$ and $|\nabla u|(L,0)$. If $L=l$, then $\Omega$ and $u$ are symmetric w.r.t. the line $\{y=x\}$, and hence 
\begin{align*}
|\nabla u|(0,l)=|\nabla u|(L,0) \text{ whenever } L=l.	
\end{align*} 
Now we consider the case $L>l$ only. 
To do this, we let $T_{\lambda}=\{y=x-\lambda\}$ be the line passing through the point $(\lambda, 0)$ with slope 1, and let ${D}_{\lambda}$ be the region in $\Omega$ below the line $T_{\lambda}$. Then the reflection of $D_\lambda$ w.r.t. $T_{\lambda}$ is also contained in $\Omega$ for $L-l\leq\lambda< L+l$. Applying the well-known moving plane methods \cite{GNN, BN91}, one can deduce that 
\begin{align*}
w^{\lambda}(x, y) = u(x, y) - u(y+\lambda, x-\lambda) < 0 \text{ for } (x, y) \in {D}_{\lambda} 
\end{align*}
for every $L-l \leq \lambda< L+l$. Recalling $w^{L-l}$ is harmonic and satisfies 
$w^{L-l}(x, y)=0$ for $x=L$ and $|y|<l$. 
The Hopf lemma implies $\partial_{x} w^{L-l}(L, y)>0$, and hence
\begin{align*}
	\partial_{x}w^{L-l}(L, y) = u_{x}(L, y) - u_{y}(y+L-l, l) > 0 \text{ for } |y|<l. 
\end{align*}
In particular, $-u_{x}(L, 0) < - u_{y}(L-l, l)$. Combining this with $- u_{y}(L-l, l) < - u_{y}(0, l)$, we conclude
\begin{align*}
	|\nabla u|(0,l) > |\nabla u|(L,0) \text{ whenever } L>l.	
\end{align*}
Therefore, the proof is finished. 
\end{proof}

\begin{remark}
	The argument also works in any dimension $n$. Hence we can conclude that the  maximal norm of gradient of torsion function on rectangle in $\R^{n}$ is located only at the center of the faces of largest $(n-1)$-volume.
\end{remark}

\section{Further questions}
In this section, we mention some further questions.

\begin{question}
\label{q4}
    Let $\Omega$ be an axially symmetric convex planar domain centered at the origin and the lengths of the two axes are different. If $|\nabla u|(p)=\max_{\partial \Omega}|\nabla u|$, then is it true that $p$ can never be on the long axis? Is it true that $p$ can never be the point of maximal distance to the origin?
\end{question}

If the answer to Question \ref{q4} is positive, then a natural further question is as follows:
\begin{question}
\label{q5}
    Let $\Omega$ be an axially symmetric convex planar domain centered at the origin which is not a disk.  Let $p$ be the point at which $|\nabla u|$ takes its maximum value over $\partial \Omega$. Restricted in the first quadrant, if $\hat{x}$ is the unique point of maximal distance to the origin and $\check{x}$ is the unique point of minimal distance to the origin. Does the following inequality always hold?
    \begin{align*}
        |p-\check{x}|\le |p-\hat{x}|?
    \end{align*}
   Does the strict inequality hold?  
\end{question}
If the answer to Question \ref{q5} is positive, then we can say that for an axially symmetric planar convex domain, even if the point of maximal stress can be far away from contact points of largest inscribed circle, it is at least closer to contact points of largest inscribe circle than to the points of maximal distance. 

The following question is similar:

\begin{question}
\label{q6}
    Let $\Omega$ be an axially symmetric convex planar domain centered at the origin which is not a disk.  Let $p$ be the point at which $|\nabla u|$ takes its maximum value over $\partial \Omega$. Restricted in the first quadrant, if $\bar{x}$ is the unique point of maximal curvature and $\tilde{x}$ is the unique point of minimal curvature. Does the following inequality always hold?
    \begin{align*}
        |p-\tilde{x}|\le |p-\bar{x}|?
    \end{align*}
   Does the strict inequality hold?  
\end{question}


Next, we let $\Omega$ be an axially symmetric planar convex domain, $\rm{diam}(\Omega)$ be the diameter of $\Omega$, which means the largest distance between two points on $\partial \Omega$. We let $\mathcal{P}_\Omega$ be the set of points of maximal stress, and $\mathcal{C}_\Omega$  be the set of contact points of largest inscribed circle in $\Omega$. Let $d_\Omega$ be the  distance between $\mathcal{P}_\Omega$ and $\mathcal{C}_\Omega$, and we define
\begin{align}
\label{measure}
R_\Omega: =\frac{d_\Omega}{\rm{diam}(\Omega)}.
\end{align}
The quantity \eqref{measure} thus measures how close between the points of maximal stress and contact points of largest inscribed circle. It would be then very interesting to answer the following question.
\begin{question}
    \label{q7}
Among all axially symmetric planar convex domains $\Omega$, can we obtain a sharp upper bound of $R_\Omega$? What if we remove the axially symmetry condition? 
\end{question}

One can also ask a similar question in terms of curvature. 

Last, the following question is also open to us.
\begin{question}
    Let $\Omega$ be an axially symmetric planar convex domain and $u$ be the torsion function. What is the necessary and sufficient geometrical condition for $\Omega$ such that $|\nabla u|$ is an increasing function along the boundary in the first quadrant.
\end{question}

\section{acknowledgement} 
We would like to thank Professor Guido Sweers for very helpful comments.

\end{document}